\documentclass[12pt]{article}
\usepackage{amsmath,amssymb}
\usepackage[T1]{fontenc}
\usepackage{latexsym}
\usepackage{epsfig}
\usepackage{psfrag}
\usepackage[utf8]{inputenc}
\usepackage[a4paper, top=3cm, bottom=3cm ]{geometry}
\newtheorem{theorem}{Theorem}[section]

\newtheorem{lemma}{Lemma}[section]

\newtheorem{pr}{Proposition}[section]

\title
{THE BERGMAN KERNEL FOR INTERSECTION OF TWO COMPLEX ELLIPSOIDS}

\author{\normalsize Tomasz Beberok \\
\small Faculty of Mathematics and Computer Science, Jagiellonian University,\\
\small Lojasiewicza 6, 30-048 Krakow, Poland \\}

\date{}

\begin{document}

\begin{center}
  \textbf{The Bergman kernel for intersection of two complex ellipsoids}
\end{center}
\vskip1em
\begin{center}
  Tomasz Beberok
\end{center}

\vskip3em

In this paper we obtain the closed forms of some hypergeometric functions. As an application, we obtain the explicit forms of the Bergman kernel functions for intersection of two complex ellipsoids $\{z \in \mathbb{C}^3 \colon |z_1|^p + |z_2|^q < 1, \quad |z_1|^p + |z_3|^r < 1\}$. We consider cases $p=6, q= r= 2$ and $p=q=r=2$. We also investigate the Lu Qi-Keng problem for $p=q=r=2$.
\vskip1em


\section{Introduction}

In 1921, S. Bergman introduced a kernel function, which is now known as the Bergman kernel function. It is well known that there exists a unique Bergman kernel function for each bounded domain in $\mathbb{C}^n$. Computation of the Bergman kernel function by explicit formulas is an important research direction in several
complex variables. Let $D$ be a bounded domain in $\mathbb{C}^n$. The Bergman space $L^2_a(D)$ is the space of all square integrable holomorphic functions on $D$. Then the Bergman kernel $K_D(z,w)$ is defined \cite{BE} by
\begin{align*}
K_D(z,w)= \sum_{j=0}^{\infty} \phi_j(z) \overline{\phi_j(w)}, \quad (z,w) \in D \times D,
\end{align*}
where $\{\phi_j(·) \colon  j = 0, 1, 2, . . .\}$ is a complete orthonormal basis for $L^2_a(D)$. If D is the Hermitian unit ball $B_n$ defined by
\begin{align*}
B_n=\{z \in \mathbb{C}^n \colon |z_1|^2 + |z_2|^2 + \ldots + |z_n|^2 < 1 \},
\end{align*}
It is easy to see that $z^{\alpha}$, $\alpha \in (\mathbb{Z}_{+})^n $ form an orthogonal basis of $L^2_a(B_n)$. A direct computation shows that $\| z^{\alpha} \| = \sqrt{ \frac{\alpha ! \pi^n }{(n + | \alpha|)!} }$. So the functions $\varphi_{\alpha} = \sqrt{ \frac{(n + | \alpha|)!}{\alpha ! \pi^n} } z^{\alpha}$, $\alpha \in (\mathbb{Z}_{+})^n $, form an orthonormal basis of $L^2_a(D)$. An easy computation gives:

\begin{align*}
K_{B_n}(z,w)=\frac{n!}{\pi^n} \frac{1}{(1- \langle z,w \rangle)^{n+1}},
\end{align*}
where $\langle z,w \rangle := z_1\overline{w}_1 + \ldots + z_n\overline{w}_n$. J.-D. Park in \cite{Park1} compute Bergman kernel for nonhomogeneous domain $$D_{q_1,q_2}=\{(z_1, z_2) \in \mathbb{C}^2 : |z_1|^{4/q_1} + |z_2|^{4/q_2} < 1\}$$ for any positive integers $q_1$ and $q_2$. The goal of this paper is to give Bergman kernel for $\{z \in \mathbb{C}^3 \colon |z_1|^p + |z_2|^q < 1, \quad |z_1|^p + |z_3|^r < 1\}$ in cases when $p=6, q= r= 2$ or $p=q=r=2$.

\section{\bf Main results}
The following are the main theorems of this paper.

\begin{theorem}\label{th1}
The Bergman kernel for $D_1= \{(z_1, z_2, z_3) \in \mathbb{C}^3 : |z_1|^2 + |z_2|^2 <1 , \quad  |z_1|^2 + |z_3|^2 < 1\}$ is given by
\begin{align*}
K_{D_1}((z_1,z_2,z_3),(w_1,w_2,w_3))= \frac{3 -6\nu_1 + 3 \nu_1^2 + \nu_1(\nu_2 + \nu_3) - \nu_2 - \nu_3 -\nu_2 \nu_3}{\pi^3 (1- \nu_1- \nu_2 )^3 (1- \nu_1 - \nu_3)^3},
\end{align*}

where $\nu_i= z_i \overline{w_i}$ for $i=1,2,3$.
\end{theorem}

\begin{theorem}
The Bergman kernel for $D_2= \{(z_1, z_2, z_3) \in \mathbb{C}^3 : |z_1|^6 + |z_2|^2 <1 , \quad  |z_1|^6 + |z_3|^2 < 1\}$ is given by
\begin{align*}
K_{D_2}&((z_1,z_2,z_3),(w_1,w_2,w_3))=\\  &\frac{3}{2 \pi^3 } \frac{\partial^2}{\partial \nu_2 \partial \nu_3} \left\{ \frac{ \nu_1 + 2}{2(1- \nu_1^3- \nu_3)^2} \right.  + \frac{\nu_2 \nu_3 (V_1(\nu_1,\nu_2,\nu_3) + \nu_1 V_2(\nu_1,\nu_2,\nu_3)) }{\left(1- \nu_1^3\right)^2 \left(1- \nu_1^3- \nu_2 \right)^2 \left(1- \nu_1^3- \nu_3 \right)^2} \\ &+  \frac{\frac{2 \nu_3 }{\sqrt[3]{1- \nu_2}} -\frac{2 \nu_2}{\sqrt[3]{1-\nu_3}} + (2+ \nu_1)( \nu_2 - \nu_3)   - \frac{\nu_2 \nu_1 }{(1- \nu_3)^{2/3}} + \frac{\nu_3 \nu_1}{(1-\nu_2)^{2/3}} }{3 \left(1- \nu_1^3- \nu_3\right) (\nu_3- \nu_2)}  \\ &+ \frac{2 \nu_2  (1- \nu_3)^{2/3} -2 \nu_3 (1- \nu_2)^{2/3} - \nu_3 \sqrt[3]{1- \nu_2} + \nu_2 \sqrt[3]{1- \nu_3} }{2  \left(1 - \nu_1^3 - \nu_3 \right)^2 (\nu_3 - \nu_2)}\\ &+ \frac{\nu_3  \left(1 + \nu_1 - \nu_1 \sqrt[3]{1- \nu_2 } - (1- \nu_2)^{2/3} \right) \left(2 \nu_1^3 + \nu_2 + \nu_3 - 2\right)}{  \left(1 - \nu_1^3 -  \nu_2 \right)^2 \left(1- \nu_1^3 - \nu_3 \right)^2} \\ &+ \left. \frac{ \nu_3 ( (3 \nu_1 + 2) (1- \nu_2) -2(1-\nu_2)^{2/3}  -3 \nu_1 \sqrt[3]{1-\nu_2})}{3 (1-\nu_2)  \left(1- \nu_1^3- \nu_2\right) \left(1 -\nu_1^3- \nu_3\right)} \right\}  ,
\end{align*}

where $\nu_i= z_i \overline{w_i}$ for $i=1,2,3$, and
\begin{align*}
V_1(\nu_1,\nu_2,\nu_3)=&6 \nu_1^9 + \nu_1^6 (6 \nu_2 +6 \nu_3 -7)+ \nu_1^3 (\nu_2 (6 \nu_3 -2)-2 (\nu_3 +2 )) \\ &+  \nu_2 (3 \nu_3 - 4)-4 \nu_3 + 5  + \nu_1(\nu_2 (6 \nu_3 -7)-7 \nu_3 +8)\\ V_2(\nu_1,\nu_2,\nu_3)=&3 \nu_1^9+ \nu_1^6 (3 \nu_2 +3 \nu_3+ 2)+ \nu_1^3 (\nu_2 (3 \nu_3 +4)+4 \nu_3 - 13) \\&+ 3 \nu_1^7+2 \nu_1^4 ( \nu_2 + \nu_3 -3)+ \nu_1 \nu_2 (\nu_3 -2)-2 \nu_1 \nu_3 + 3\nu_1 \end{align*}
\end{theorem}

\begin{theorem}
The Bergman kernel for $D_3= \{(z_1, z_2, z_3) \in \mathbb{C}^3 : |z_1|^2 + |z_2|^2 <1 , \quad  |z_1|^4 + |z_3|^2 < |z_1|^2 \}$ is given by
$$K_{D_3}(z,w)=\frac{\nu_1 (\nu_3 (\nu_1-\nu_2-1)+(\nu_1-1) \nu_1 (3 \nu_1+\nu_3-3))}{ \pi^3(\nu_1+\nu_2-1)^3 ((\nu_1-1) \nu_1+\nu_3)^3},$$
where $z=(z_1,z_2,z_3)$, $w=(w_1,w_2,w_3)$ and $ $$\nu_i= z_i \overline{w_i}$ for $i=1,2,3$.
\end{theorem}

\begin{theorem}
The Bergman kernel for domain $D_4$ defined $$\left\{z \in \mathbb{C}^4 : |z_1|^2 + |z_2|^2 + |z_3|^2<1 , \,  \left(|z_1|^2 + |z_2|^2\right)^2 + |z_4|^2 < |z_1|^2 + |z_2|^2 \right\}$$ is given by
\begin{align*}
K_{D_4}&((z_1,z_2,z_3,z_4),(w_1,w_2,w_3,w_4)) = \frac{1}{ \pi^4 } \frac{\partial^2}{\partial \nu_3 \partial \nu_4} \left\{ \frac{\nu_3}{\nu_3-\nu_3^2 -\nu_4} \right. \\ &\left(\frac{ 1+\nu_1+\nu_2-2\nu_3}{  (1-\nu_1-\nu_2)^3} \right. +  \frac{\nu_4 \nu_3(\nu_4-\nu_3)}{(\nu_3-\nu_4-1) (1-\nu_1-\nu_2)^2}    \\ &+ \frac{ \nu_4 \nu_3 (2-2 \nu_1-2 \nu_2- \nu_3)}{(\nu_3-\nu_4-1) (1- \nu_1 -\nu_2)^2 (1 - \nu_1 -\nu_2 - \nu_3)^2}   \\ &+ \frac{  \nu_4 \left(6 \nu_3 (\nu_1+\nu_2-1)+6 (\nu_1+\nu_2-1)^2+2 \nu_3^2\right) }{ (\nu_1+\nu_2-1)^3 (\nu_1+\nu_2+\nu_3-1)^3}   \\&+ \frac{8 \left(W_1(\nu_1,\nu_2,\nu_3,\nu_4) + \sqrt{1-4 \nu_4} W_2(\nu_1,\nu_2,\nu_3,\nu_4) \right) }{(1-4 \nu_4)^{3/2} \nu_4  \left(\sqrt{1-4 \nu_4}-2 \nu_1 + 1\right) \left(\sqrt{1-4 \nu_4}-2 \nu_1-2 \nu_2+1\right)^3} \\   &+  \left. \left. \frac{ \nu_4 \nu_3 \left(\nu_3 - \nu_4 \right)   \left(\sqrt{1-4 \nu_4}+1\right)^2  + 4 \nu_4 }{ (\nu_3-\nu_4-1) \sqrt{1-4 \nu_4} \left(\sqrt{1-4 \nu_4}-2 \nu_1-2 \nu_2+1\right)^2} \right) \right\},
  \end{align*}
where $\nu_i= z_i \overline{w_i}$ for $i=1,2,3,4$,  and \\ $W_1(\nu_1,\nu_2,\nu_3,\nu_4)=\nu_4^2 (8 \nu_1+8 \nu_2-4 \nu_3+2) -(\nu_1-1) (\nu_1+\nu_2+2 \nu_3+1) \\+\nu_4 \left(\nu_1^2 (4 \nu_3+6)+\nu_1 (4 \nu_2 \nu_3+6 \nu_2+4 \nu_3-4)-2 \nu_2 \nu_3-7 \nu_2-8 \nu_3-4\right)$, and \\ $W_2(\nu_1,\nu_2,\nu_3,\nu_4)=w (4 \nu_1 (\nu_1+\nu_2-1)-2 \nu_2 \nu_3-5 \nu_2-4 \nu_3-2)\\-(\nu_1-1) (\nu_1+\nu_2+2 \nu_3+1)$.
\end{theorem}

\section{\bf  Explicit formulas of hypergeometric functions}

A great interest in the theory of hypergeometric functions (that is, hypergeometric functions of several
variables) is motivated essentially by the fact that the solutions of many applied problems involving
(for example) partial differential equations are obtainable with the help of such hypergeometric function
(see, for details, (\cite{Multiple}, p. 47); see also other works (\cite{H1, H2}, \cite{OSS}) and the references cited therein). For instance, the energy absorbed by some non-ferromagnetic conductor sphere included in an internal
magnetic field can be calculated with the help of such functions \cite{LO}, \cite{NI}. Hypergeometric functions of
several variables are used in physical and quantum chemical applications as well \cite{MA}, \cite{SE}. Especially,
many problems in gas dynamics lead to solutions of degenerate second-order partial differential
equations, which are then solvable in terms of multiple hypergeometric functions. Among examples, we
can cite the problem of adiabatic flat-parallel gas flow without whirlwind, the flow problem of supersonic
current from vessel with flat walls, and a number of other problems connected with gas flow \cite{Bers}, \cite{Fran}. Multiple hypergeometric functions (that is, hypergeometric functions in several variables) occur naturally
in a wide variety of problems. In particular, one of the  Lauricella functions

\begin{align*}
F_8(a,b_1,b_2,b_3;c_1,c_2;x,y,z)= \sum_{m,n,p=0}^{\infty} \frac{(a)_{m+n+p} (b_1)_m (b_2)_n (b_3)_p }{(c_1)_m (c_2)_{n+p} m! n! p!} x^m y^n z^p,
\end{align*}

Appell's functions $F$, $F_1$ and $F_2$ defined by

\begin{align*}
F(a,b;c;x)= \sum_{m=0}^{\infty} \frac{(a)_{m} (b)_m   }{ (c)_{m} m! } x^m ,
\end{align*}

\begin{align*}
F_1(a,b_1,b_2;c;x,y)= \sum_{m,n=0}^{\infty} \frac{(a)_{m+n} (b_1)_m (b_2)_n  }{ (c)_{m+n} m! n!} x^m y^n,
\end{align*}

\begin{align*}
F_2(a,b_1,b_2;c_1,c_2;x,y)= \sum_{m,n=0}^{\infty} \frac{(a)_{m+n} (b_1)_m (b_2)_n  }{ (c_1)_{m} (c_2)_n m! n!} x^m y^n.
\end{align*}

and Horn's function $H_3$

\begin{align*}
H_3(a,b;c;x,y)= \sum_{m,n=0}^{\infty} \frac{(a)_{2m+n} (b)_n  }{ (c)_{m + n}  m! n!} x^m y^n.
\end{align*}

For function $F_1$ we have the following integral representation (see \cite{Edr})
\begin{align*}
F_1(a,b_1,b_2;c;x,y)= \frac{\Gamma(c)}{\Gamma(b_1) \Gamma(b_2) \Gamma(c- b_1 - b_2)} \\ \cdot \iint\limits_{\substack{u \geq 0 , v \geq 0 \\ u + v \leq 1   } } u^{b_1 -1} v^{b_2 - 1} (1-u-v)^{c-b_1-b_2-1} (1-ux-vy)^{-a} \,dudv ,
\end{align*}
where $\Re (b_1) > 0, \Re (b_2) > 0, \Re(c-b_1-b_2) > 0$. \newline
Picard has pointed out that $F_1$ can be represented by a single integral in the form

\begin{align*}
F_1(a,b_1,b_2;c;x,y)=& \frac{\Gamma(c)}{\Gamma(a)  \Gamma(c- a)} \\ &\cdot \int_0^1 u^{a-1} (1-u)^{c-a-1} (1-u x)^{-b_1} (1- u y)^{-b_2}\, du,
\end{align*}
where  $\Re (a) > 0, \Re(c-a) > 0$.

In \cite{decH3} presented certain interesting integral representation for Horn's $H_3$ function

\begin{align*}
H_3(a,b;c;x,y)= \frac{\Gamma(c)}{\Gamma(b)  \Gamma(c- b)} & \int_0^1 u^{b-1} (1-u)^{c-b-1} (1-u y)^{-a} \\& \cdot F\left(\frac{a}{2}, \frac{a+1}{2};c-b ; \frac{4x(1-u)}{(1 - u y)^2} \right) \, du,
\end{align*}
where  $\Re (c)  > \Re(b) > 0$.

In this section we prove following recursion formula for  $$F_8(a,1;c_1,c_2;x,y,z):=F_8(a,1,1,1;c_1,c_2;x,y,z).$$

\begin{pr}
For any $a>1$ and $ |x| + |y| <1$, $|x| + |z|<1$, we have
\begin{align*}
F_8(a,1;c_1,c_2;x,y,z)=& \frac{a-c_1 - c_2 +1}{(a-1)(1-x-z)} F_8(a-1,1;c_1,c_2;x,y,z) \\ +& \frac{c_1-1}{a-1} \frac{1}{1-x-z} F_1(a-1,1,1;c_2;y,z)\\ +&  \frac{c_2-1}{a-1} \frac{1}{1-x-z} F_2(a-1,1,1;c_1,c_2-1;x,y)
\end{align*}

\end{pr}

\begin{proof}
Using well know fact $(1)_k=k!$ we have

\begin{align*}
F_8(a,1;c_1,c_2;x,y,z)=\sum_{m,n,p=0}^{\infty} \frac{(a)_{m+n+p}  }{(c_1)_m (c_2)_{n+p} } x^m y^n z^p
\end{align*}

Next from $(s)_k=\frac{(s-1+k)(s-1)_k}{s-1}$ we have

\begin{align*}
F_8(a,1;c_1,c_2;x,y,z)=\sum_{m,n,p=0}^{\infty} \frac{(a-1+m+n+p)(a-1)_{m+n+p}  }{(a-1)(c_1)_m (c_2)_{n+p} } x^m y^n z^p
\end{align*}
After little calculation we obtain

\begin{align*}
F_8(a,1;c_1,c_2;x,y,z)=& \frac{a+1-c_1-c_2}{a-1} F_8(a-1,1;c_1,c_2;x,y,z) \\&+  \frac{c_1-1}{a-1}  \sum_{m,n,p=0}^{\infty} \frac{(a-1)_{m+n+p}  }{(c_1-1)_m (c_2)_{n+p} } x^m y^n z^p \\&+  \frac{c_2-1}{a-1}  \sum_{m,n,p=0}^{\infty} \frac{(a-1)_{m+n+p}  }{(c_1)_m (c_2-1)_{n+p} } x^m y^n z^p
\end{align*}

Now summing the m and p respectively in the second and third lines, we have

\begin{align*}
&F_8(a,1;c_1,c_2;x,y,z)= \frac{a+1-c_1-c_2}{a-1} F_8(a-1,1;c_1,c_2;x,y,z) \\&+ \frac{c_1-1}{a-1} \left(\sum_{n,p=0}^{\infty} \frac{(a-1)_{n+p}  }{ (c_2)_{n+p} } y^n z^p + \sum_{m=1,n,p=0}^{\infty} \frac{(a-1)_{m+n+p}  }{(c_1-1)_m (c_2)_{n+p} } x^m y^n z^p \right)\\&+  \frac{c_2-1}{a-1} \left( \sum_{m,n=0}^{\infty} \frac{(a-1)_{m+n}  }{(c_1)_m (c_2-1)_{n} } x^m y^n + \sum_{m,n=0,p=1}^{\infty} \frac{(a-1)_{m+n+p}  }{(c_1)_m (c_2-1)_{n+p} } x^m y^n z^p \right)
\end{align*}

Reversing the summation indexes, we can write

\begin{align*}
F_8&(a,1;c_1,c_2;x,y,z)= \frac{a+1-c_1-c_2}{a-1} F_8(a-1,1;c_1,c_2;x,y,z) \\ & + \frac{c_1-1}{a-1} \sum_{n,p=0}^{\infty} \frac{(a-1)_{n+p}  }{ (c_2)_{n+p} } y^n z^p + x F_8(a,1;c_1,c_2;x,y,z) \\ &+ \frac{c_2-1}{a-1}  \sum_{m,n=0}^{\infty} \frac{(a-1)_{m+n}  }{(c_1)_m (c_2-1)_{n} } x^m y^n + z F_8(a,1;c_1,c_2;x,y,z)
\end{align*}
which complete proof.
\end{proof}

Now we calculate formulas in special cases.

\begin{lemma}\label{lem1}
For $ |x| + |y| <1$, $|x| + |z|<1$ and $y \neq 0$,  $z \neq 0$ $ z \neq y$, we have
\begin{align*}
F_8&\left( \frac{10}{3} ,1; \frac{1}{3} ,3;x^3,y,z \right)=  \frac{ -\frac{9y}{\sqrt[3]{1-z}}+ \frac{9z}{\sqrt[3]{1-y}}-9 (z-y)}{7 y z  (z-y)M_{zx}}  + \frac{9 \left(\sqrt[3]{1-y}-1\right)}{7 y \sqrt[3]{1-y}  M_{yx} M_{zx}} \\&+ \frac{9P(x,y,z)}{14 \left(1-x^3\right)^2 M_{yx}^2 M_{zx}^2} + \frac{27 \left((1-y)^{2/3}-1\right) (M_{yx} + M_{zx} )}{14 y M_{yx}^2 M_{zx}^2} \\& -\frac{27 \left(-z(1-y)^{2/3} + y (1-z)^{2/3}-y+z\right)}{14 y z  (y-z) M_{zx}^2} ,
\end{align*}

\begin{align*}
F_8&\left( \frac{11}{3} ,1; \frac{2}{3} ,3;x^3,y,z \right)=  \frac{ -\frac{9y}{(1-z)^{2/3}}+ \frac{9z}{(1-y)^{2/3}}-9(z-y)}{40 y z (z-y)M_{zx}}  \\ &+ \frac{9G(x,y,z)}{40 \left(1-x^3\right)^2 M_{yx}^2 M_{zx}^2}   + \frac{27 \left(\sqrt[3]{1-y}-1\right) (M_{yx} + M_{zx} ) }{40 y M_{yx}^2 M_{zx}^2}  \\&-  \frac{27 \left(-z\sqrt[3]{1-y} + y \sqrt[3]{1-z}-y+z\right)}{80 y z  (y-z)M_{zx}^2}  + \frac{27 \left((1-y)^{2/3}-1\right)}{40 y (1-y)^{2/3}  M_{yx} M_{zx}},
\end{align*}

\begin{align*}
F_8\left( 4 ,1; 1 ,3;x^3,y,z \right)=  \frac{3 x^6+2 x^3 (y+z-3)+yz-2y-2 z+3}{3 \left(x^3-1\right)^2 M_{yx}^2 M_{zx}^2},
\end{align*}
where $P(x,y,z)=2 x^9+x^6 (2 y+2 z+3)+2 x^3 (y z+y+z-6)+y z-4y-4 z+7$, \\ $G(x,y,z)=x^9+x^6 (y+z+6)+x^3 (y z+4y+4 z-15)+2 y  z-5y-5 z+8$ and  $M_{tx}=1-x^3-t$.

If $y=0$ then we have

\begin{align*}
F_8&\left( \frac{10}{3} ,1; \frac{1}{3} ,3;x^3,y,z \right)=F_2\left( \frac{10}{3} ,1,1; \frac{1}{3} ,3;x^3,z \right)= \\ &  \frac{x^3 \left(2 x^3 \left(18-x^3-z\right)+16 z-21\right)+4 z-13}{14 \left(x^3-1\right)^3 M_{zx}^2}\\&+ \frac{27 (1-z)^{2/3}+18 z-27}{14 z^2 M_{zx}^2}+\frac{3 z \sqrt[3]{1-z}+9 \sqrt[3]{1-z}-9}{7 z^2 \sqrt[3]{1-z}  M_{zx}},
\end{align*}

\begin{align*}
F_8&\left( \frac{11}{3} ,1; \frac{2}{3} ,3;x^3,y,z \right)= F_2\left( \frac{11}{3} ,1,1; \frac{2}{3} ,3;x^3,z \right)= \\&\frac{-6 x^9+x^6 (45-6 z)+x^3 (30 z+9)+6 (5 z-8)}{40 \left(x^3-1\right)^3 M_{zx}^2} \\&+\frac{27 \sqrt[3]{1-z}+9z-27}{40 z^2 M_{zx}^2} +  \frac{6 z (1-z)^{2/3}+9 (1-z)^{2/3}-9}{40z^2 (1-z)^{2/3}  M_{zx}},
\end{align*}
where as before $M_{tx}=1-x^3-t$.
In the case when $z=y=0$, we have

\begin{align*}
F_8\left( \frac{10}{3} ,1; \frac{1}{3} ,3;x^3,y,z \right)= F \left( \frac{10}{3} ,1; \frac{1}{3};x^3 \right)=  \frac{4 x^9-21 x^6+84 x^3+14}{14 \left(1-x^3\right)^4}
\end{align*}

\begin{align*}
F_8\left( \frac{11}{3} ,1; \frac{2}{3} ,3;x^3,y,z \right)= F \left( \frac{11}{3} ,1; \frac{2}{3};x^3 \right)=  \frac{5 x^9-24 x^6+60 x^3+40}{40 \left(1-x^3\right)^4}
\end{align*}

If $y=z$ and $z \neq 0$ then we have

\begin{align*}
F_8\left( \frac{10}{3} ,1; \frac{1}{3} ,3;x^3,z,z \right)= \lim_{y \rightarrow z} F_8\left( \frac{10}{3} ,1; \frac{1}{3} ,3;x^3,y,z \right)
\end{align*}

\begin{align*}
F_8\left( \frac{11}{3} ,1; \frac{2}{3} ,3;x^3,z,z \right)= \lim_{y \rightarrow z} F_8\left( \frac{11}{3} ,1; \frac{2}{3} ,3;x^3,y,z \right)
\end{align*}

\end{lemma}

\begin{proof}
Using twice recursion formula for $F_8$ function, we have
\begin{align*}
F_8\left( \frac{10}{3} ,1; \frac{1}{3} ,3;x^3,y,z \right)=& \frac{-3F_1\left(\frac{4}{3}, 1, 1; 3; y,z \right)}{14(1- x^3 - z )^2 }  +  \frac{9F_2\left(\frac{4}{3}, 1, 1; \frac{1}{3},2; x^3,y \right)}{14(1- x^3 - z )^2 } \\ &+ \frac{-2F_1\left(\frac{7}{3}, 1, 1; 3; y,z \right)}{7(1- x^3 - z ) }  +  \frac{6F_2\left(\frac{7}{3}, 1, 1; \frac{1}{3},2; x^3,y \right)}{7(1- x^3 - z ) }
\end{align*}
Now using recursion formula for $F_2$ (see \cite{Park1} or more general version  \cite{Park2}), we obtain

\begin{align*}
F_8&\left( \frac{10}{3} ,1; \frac{1}{3} ,3;x^3,y,z \right)= \frac{-3F_1\left(\frac{4}{3}, 1, 1; 3; y,z \right)}{14(1- x^3 - z )^2 } + \frac{-2F_1\left(\frac{7}{3}, 1, 1; 3; y,z \right)}{7(1- x^3 - z ) } \\ &-\frac{18 F\left(\frac{1}{3},1; 2; y \right) }{14 \left(-x^3-y+1\right) \left(-x^3-z+1\right)^2} + \frac{27 F\left(\frac{1}{3},1; \frac{1}{3}; x^3 \right) }{14  \left(-x^3-y+1\right) \left(-x^3-z+1\right)^2} \\ &+ \frac{27 \, F\left(\frac{1}{3},1;\frac{1}{3};x^3\right)}{14 \left(-x^3-y+1\right)^2 \left(-x^3-z+1\right)}-\frac{9 \, F\left(\frac{1}{3},1;2;y\right)}{7 \left(-x^3-y+1\right)^2 \left(-x^3-z+1\right)}\\ &+ \frac{9 \, F\left(\frac{4}{3},1;\frac{1}{3};x^3\right)}{14 \left(-x^3-y+1\right) \left(-x^3-z+1\right)}-\frac{3 \, F\left(\frac{4}{3},1;2;y\right)}{7 \left(-x^3-y+1\right) \left(-x^3-z+1\right)}
\end{align*}

After some calculations using integral representation for $F_1$ function and  well known formulas $ F\left(\frac{1}{3},1;2;w\right) = \frac{3 \left(1-(1-w)^{2/3}\right)}{2 w} ,$  $ F\left(\frac{4}{3},1;2;w\right) = \frac{3 \left(1- \sqrt[3]{1-w}\right)}{\sqrt[3]{1-w} w}$ and \\ $F\left(\frac{4}{3},1;\frac{1}{3};w\right)  = \frac{2 w+1}{(1-w)^2}$  we obtain desired result. The proof for the other formulas in similar to above.

\end{proof}
To prove Bergman kernel formula for domain $D_4$ we need the following lemma.

It is possible that the following sum is represented by a hypergeometric function or is related to a combination of certain hypergeometric functions, but from the point of view of this work most interesting is its calculation.
\begin{lemma}\label{lemd3} For $|x| + |y| + |z| <1$, $|4w|<1$, $w\neq 0$, $w \neq z - z^2$ and $2|x| + 2|y| - |\sqrt{1-4w}|<1$, we have
  \begin{align*}
      &\sum_{m,n,k,l=0}^{\infty}  \frac{ \Gamma\left( m + n + 2  \right) \Gamma\left( m + n + k + 2l + 6  \right) }{ m! n! \Gamma\left( m + n + l + 3 \right) \Gamma\left( k + l + 3  \right) } x^m y^n z^k w^l = \\ &\frac{ z(2-2 x-2 y-z)}{(z-z^2 -w)(z-w-1) (1- x -y)^2 (1 - x -y - z)^2} \\ &+ \frac{2z}{w (z -z^2 - w)(x+y-1)^3}+  \frac{z(w-z)}{(z-z^2 - w)(z-w-1) (1-x-y)^2}  \\ &+  \frac{z \left(z-w \right)   \left(\sqrt{1-4 w}+1\right)^2  + 4}{ (z-z^2 - w)(z-w-1) \sqrt{1-4 w} \left(\sqrt{1-4 w}-2 x-2 y+1\right)^2}                \\ &+ \frac{6 z (x+y-1)+6 (x+y-1)^2+2 z^2}{(z-z^2 - w) (x+y-1)^3 (x+y+z-1)^3} + \frac{1+x+y}{w (z-z^2 - w) (1-x-y)^3} \\ &+ \frac{8 \left(W_1(x,y,z,w) + \sqrt{1-4 w} W_2(x,y,z,w) \right) }{(z-z^2 - w)(1-4 w)^{3/2} w  \left(\sqrt{1-4 w}-2 x+1\right) \left(\sqrt{1-4 w}-2 x-2 y+1\right)^3},
  \end{align*}
where $W_1(x,y,z,w)=w^2 (8 x+8 y-4 z+2)-(x-1) (x+y+2 z+1)\\+w \left(x^2 (4 z+6)+x (4 y z+6 y+4 z-4)-2 y z-7 y-8 z-4\right)$, and \\$W_2(x,y,z,w)=w (4 x (x+y-1)-2 y z-5 y-4 z-2)-(x-1) (x+y+2 z+1)$
\end{lemma}
\begin{proof}
Since $\Gamma(a+1)=a \Gamma(a)$
   \begin{align*}
    S:=\sum_{m,n,k,l=0}^{\infty}  \frac{ \Gamma\left( m + n + 2  \right) \Gamma\left( m + n + k + 2l + 6  \right) }{ m! n! \Gamma\left( m + n + l + 3 \right) \Gamma\left( k + l + 3  \right) } x^m y^n z^k w^l= \\ \sum_{m,n,k,l=0}^{\infty}  \frac{\left( m + n + k + 2l + 5  \right) \Gamma\left( m + n + 2  \right) \Gamma\left( m + n + k + 2l + 5  \right) }{ m! n! \Gamma\left( m + n + l + 3 \right) \Gamma\left( k + l + 3  \right) } x^m y^n z^k w^l.
   \end{align*}
   Hence
   \begin{align*}   S=& S_1:=\sum_{m,n,k,l=0}^{\infty}  \frac{ \Gamma\left( m + n + 2  \right) \Gamma\left( m + n + k + 2l + 5  \right) }{ m! n! \Gamma\left( m + n + l + 3 \right) \Gamma\left( k + l + 3  \right) } x^m y^n z^k w^l  \\ &+ S_2:=\sum_{m,n,k,l=0}^{\infty}  \frac{ \Gamma\left( m + n + 2  \right) \Gamma\left( m + n + k + 2l + 5  \right) }{ m! n! \Gamma\left( m + n + l + 2 \right) \Gamma\left( k + l + 3  \right) } x^m y^n z^k w^l  \\ &+ S_3:=\sum_{m,n,k,l=0}^{\infty}  \frac{ \Gamma\left( m + n + 2  \right) \Gamma\left( m + n + k + 2l + 5  \right) }{ m! n! \Gamma\left( m + n + l + 3 \right) \Gamma\left( k + l + 2  \right) } x^m y^n z^k w^l.
   \end{align*}
\noindent By proceeding in a similar manner as in the case of the sum of $ S $, we have
          \begin{align*}
       S_1=\sum_{m,n,k,l=0}^{\infty}  \frac{ \Gamma\left( m + n + 2  \right) \Gamma\left( m + n + k + 2l + 5  \right) }{ m! n! \Gamma\left( m + n + l + 3 \right) \Gamma\left( k + l + 3  \right) } x^m y^n z^k w^l= \\ \sum_{m,n,k,l=0}^{\infty}  \frac{ \Gamma\left( m + n + 2  \right) \Gamma\left( m + n + k + 2l + 4  \right) }{ m! n! \Gamma\left( m + n + l + 2 \right) \Gamma\left( k + l + 3  \right) } x^m y^n z^k w^l  \\ + \sum_{m,n,k,l=0}^{\infty}  \frac{ \Gamma\left( m + n + 2  \right) \Gamma\left( m + n + k + 2l + 4  \right) }{ m! n! \Gamma\left( m + n + l + 3 \right) \Gamma\left( k + l + 2  \right) } x^m y^n z^k w^l.
  \end{align*}
   Using the identity $\Gamma(a+1)=a \Gamma(a)$, after a little simplification, we obtain
        \begin{align*}
       \left(1-\frac{w}{z}-z\right)S_1=&\frac{1}{z} \sum_{m,n,k=0}^{\infty}  \frac{ \Gamma\left( m + n + 2  \right) \Gamma\left( m + n + k  + 3  \right) }{ m! n! \Gamma\left( m + n  + 2 \right) \Gamma\left( k  + 2  \right) } x^m y^n z^k  \\ &- \frac{1}{z} \sum_{m,n,l=0}^{\infty}  \frac{ \Gamma\left( m + n + 2  \right) \Gamma\left( m + n  + 2l + 3  \right) }{ m! n! \Gamma\left( m + n + l + 2 \right) \Gamma\left(   l + 2  \right) } x^m y^n  w^l  \\ &+ \sum_{m,n,l=0}^{\infty}  \frac{ \Gamma\left( m + n + 2  \right) \Gamma\left( m + n  + 2l + 4  \right) }{ m! n! \Gamma\left( m + n + l + 3 \right) \Gamma\left(  l + 2  \right) } x^m y^n  w^l.
  \end{align*}
       After some calculations, we obtain
       \begin{align*}
       \left(1-\frac{w}{z}-z\right)S_1&=\frac{2-2 x-2 y-z}{z (1- x -y)^2 (1 - x -y - z)^2}  \\ &+(1- 1/z) \sum_{m,n,l=0}^{\infty}  \frac{ \Gamma\left( m + n + 2  \right) \Gamma\left( m + n  + 2l + 3  \right) }{ m! n! \Gamma\left( m + n + l + 2 \right) \Gamma\left(   l + 2  \right) } x^m y^n  w^l  \\ &+ \sum_{m,n,l=0}^{\infty}  \frac{ \Gamma\left( m + n + 2  \right) \Gamma\left( m + n  + 2l + 3  \right) }{ m! n! \Gamma\left( m + n + l + 3 \right) \Gamma\left(  l + 1  \right) } x^m y^n  w^l.
  \end{align*} Shifting (changing) the summation index $l$, we have

       \begin{align*}
       \left(1-\frac{w}{z}-z\right)&S_1=\frac{2-2 x-2 y-z}{z (1- x -y)^2 (1 - x -y - z)^2}  \\ &+ \left(\frac{1}{w}- \frac{1}{zw} \right) \sum_{m,n,l=0}^{\infty}  \frac{ \Gamma\left( m + n + 2  \right) \Gamma\left( m + n  + 2l + 1  \right) }{ m! n! \Gamma\left( m + n + l + 1 \right) \Gamma\left(   l + 1  \right) } x^m y^n  w^l \\ &- \left(\frac{1}{w}- \frac{1}{zw} \right) \sum_{m,n=0}^{\infty}  \frac{ \Gamma\left( m + n + 2  \right)  }{ m! n!  } x^m y^n \\ &+ \sum_{m,n,l=0}^{\infty}  \frac{ \Gamma\left( m + n + 2  \right) \Gamma\left( m + n  + 2l + 3  \right) }{ m! n! \Gamma\left( m + n + l + 3 \right) \Gamma\left(  l + 1  \right) } x^m y^n  w^l.
  \end{align*}
\noindent   Sum out of $l$
\begin{align*}
       &\left(1-\frac{w}{z}-z\right)S_1=\frac{2-2 x-2 y-z}{z (1- x -y)^2 (1 - x -y - z)^2}  \\ &+  \sum_{m,n=0}^{\infty}  \frac{ (z-1) \Gamma\left( m + n + 2  \right) F(\frac{m + n + 1}{2}, \frac{m + n + 2}{2};m + n + 1;4w )  }{ z w  m! n! } x^m y^n   \\ &+  \sum_{m,n=0}^{\infty}  \frac{(1-z) \Gamma\left( m + n + 2  \right)  }{z w  m! n!  } x^m y^n \\ &+ \sum_{m,n=0}^{\infty}  \frac{ \Gamma\left( m + n + 2  \right) F(\frac{m + n + 3}{2}, \frac{m + n + 4}{2};m + n + 3;4w )  }{ m! n! } x^m y^n.
  \end{align*}
 Using the following well know formula
 $$ F\left( a, a+ \frac{1}{2}; 2a; z \right) = \frac{2^{2a-1}}{ \sqrt{1-z} ( \sqrt{1-z} +1 )^{2a-1} },$$
\noindent we obtain
\begin{align*}
       \left(1-\frac{w}{z}-z\right)S_1=&\frac{2-2 x-2 y-z}{z (1- x -y)^2 (1 - x -y - z)^2}  \\ &+ \left(\frac{1}{w}- \frac{1}{zw} \right) \sum_{m,n=0}^{\infty}  \frac{ \Gamma\left( m + n + 2  \right) 2^{m+n} }{ m! n! \sqrt{1-4w} ( \sqrt{1-4w} +1 )^{m+n} } x^m y^n   \\ &- \left(\frac{1}{w}- \frac{1}{zw} \right) \sum_{m,n=0}^{\infty}  \frac{ \Gamma\left( m + n + 2  \right)  }{ m! n!  } x^m y^n \\ &+ \sum_{m,n=0}^{\infty}  \frac{ \Gamma\left( m + n + 2  \right) 2^{m+n+2} }{ m! n! \sqrt{1-4w} ( \sqrt{1-4w} +1 )^{m+n+2} } x^m y^n.
  \end{align*}
Sum out of  $m$ i $n$ variables, we get
\begin{align*}
       \left(1-\frac{w}{z}-z\right)S_1=&\frac{2-2 x-2 y-z}{z (1- x -y)^2 (1 - x -y - z)^2}  - \left(\frac{1}{w}- \frac{1}{zw} \right) \frac{1}{(1-x-y)^2} \\ &+  \frac{\left(\frac{1}{w}- \frac{1}{zw} \right)   \left(\sqrt{1-4 w}+1\right)^2  + 4}{\sqrt{1-4 w} \left(\sqrt{1-4 w}-2 x-2 y+1\right)^2}.
  \end{align*}
\noindent Hence
\begin{align*}
       S_1=& \frac{2-2 x-2 y-z}{(z-w-1) (1- x -y)^2 (1 - x -y - z)^2} -  \frac{z-w}{(z-w-1) (1-x-y)^2}  \\ &+  \frac{\left(z-w \right)   \left(\sqrt{1-4 w}+1\right)^2  + 4}{ (z-w-1) \sqrt{1-4 w} \left(\sqrt{1-4 w}-2 x-2 y+1\right)^2}.
  \end{align*}
\noindent Now analogous maneuvers for sum $S_2$ lead us to
\begin{align*}
S_2=&\frac{w}{z}S + \frac{1}{z} \sum_{m,n,k=0}^{\infty}  \frac{  \Gamma\left( m + n + k + 4  \right) }{ m! n! \Gamma\left( k + 2  \right) } x^m y^n z^k \\ &- \frac{1}{z} \sum_{m,n,l=0}^{\infty}  \frac{ \Gamma\left( m + n + 2  \right) \Gamma\left( m + n  + 2l + 4  \right) }{ m! n! \Gamma\left( m + n + l + 2 \right) \Gamma\left(  l + 2  \right) } x^m y^n  w^l.
\end{align*}
\noindent Similar as in  $S_1$ case, we have
\begin{align*}
S_2=&\frac{w}{z}S + \frac{6 z (x+y-1)+6 (x+y-1)^2+2 z^2}{z (x+y-1)^3 (x+y+z-1)^3} \\ &- \frac{1}{z w} \sum_{m,n,l=0}^{\infty}  \frac{ \Gamma\left( m + n + 2  \right) \Gamma\left( m + n  + 2l + 2  \right) }{ m! n! \Gamma\left( m + n + l + 1 \right) \Gamma\left(  l + 1  \right) } x^m y^n  w^l \\ &+ \frac{1}{z w} \sum_{m,n=0}^{\infty}  \frac{ \Gamma\left( m + n + 2  \right) \Gamma\left( m + n + 2  \right) }{ m! n! \Gamma\left( m + n + 1 \right)  } x^m y^n.
\end{align*}
\noindent Sum out of $l$
\begin{align*}
S_2=& \frac{w}{z}S + \frac{6 z (x+y-1)+6 (x+y-1)^2+2 z^2}{z (x+y-1)^3 (x+y+z-1)^3} \\ &- \frac{1}{z w} \sum_{m,n=0}^{\infty}  \frac{ \left( \Gamma ( m + n + 2) \right)^2  F\left(\frac{ m + n + 2}{2},\frac{ m + n + 3}{2}; m + n + 1;4 w\right) }{ m! n! \Gamma ( m + n + 1) } x^m y^n   \\ &+ \frac{1}{z w} \sum_{m,n,l=0}^{\infty}  \frac{ \left( m + n + 1  \right) \Gamma\left( m + n + 2  \right) }{ m! n!   } x^m y^n.
\end{align*}
\noindent Using the following formula $$ F\left( \frac{a+2}{2}, \frac{a+3}{2}; a+1; z \right) = \frac{2^{a}((a-1) (\sqrt{1-z}-1) + az) }{(a+1)z (1-z)^{3/2} ( \sqrt{1-z} +1)^{a-1} }, $$
\noindent we get
\begin{align*}
&S_2= \frac{w}{z}S + \frac{6 z (x+y-1)+6 (x+y-1)^2+2 z^2}{z (x+y-1)^3 (x+y+z-1)^3} \\ &-  \sum_{m,n=0}^{\infty}  \frac{  \Gamma ( m + n + 2) ((m + n -1) (\sqrt{1-4w}-1) + (m + n)4w)  }{ m! n! 4z w^2 (1-4w)^{3/2} ( \sqrt{1-4w} +1)^{m+n-1}  } (2x)^m (2y)^n   \\ &+ \frac{1}{z w} \sum_{m,n,l=0}^{\infty}  \frac{ \left( m + n + 1  \right) \Gamma\left( m + n + 2  \right) }{ m! n!   } x^m y^n.
\end{align*}
\noindent Sum out of $m$ and  $n$ variables, we have
\begin{align*}
S_2=& \frac{w}{z}S + \frac{6 z (x+y-1)+6 (x+y-1)^2+2 z^2}{z (x+y-1)^3 (x+y+z-1)^3} + \frac{1+x+y}{w z (1-x-y)^3} \\ &- \frac{3 \left(\sqrt{1-4 w}-1\right) \left(\sqrt{1-4 w}+1\right)^3 x (x+y)}{(1-4 w)^{3/2} w^2 z \left(\sqrt{1-4 w}-2 x+1\right) \left(\sqrt{1-4 w}-2 x-2 y+1\right)^3} \\ &+ \frac{\left(\sqrt{1-4 w}+1\right)^3 \left(\sqrt{1-4 w}-8 x^2+1\right)}{(1-4 w)^{3/2} w z \left(\sqrt{1-4 w}-2 x+1\right) \left(\sqrt{1-4 w}-2 x-2 y+1\right)^3} \\ &+ \frac{\left(\sqrt{1-4 w}+1\right)^3 \left(4 x \left(\sqrt{1-4 w}-2 y-1\right)+4 \sqrt{1-4 w} y-2 y\right)}{(1-4 w)^{3/2} w z \left(\sqrt{1-4 w}-2 x+1\right) \left(\sqrt{1-4 w}-2 x-2 y+1\right)^3}.
\end{align*}
\noindent In $S_3$ case, we have
\begin{align*}
S_3=& z S - \frac{1}{w} \sum_{m,n=0}^{\infty}  \frac{  \Gamma\left( m + n + 3  \right) }{ m! n! } x^m y^n \\ &+  \frac{1}{w}\sum_{m,n,l=0}^{\infty}  \frac{ \Gamma\left( m + n + 2  \right) \Gamma\left( m + n + 2l + 3  \right) }{ m! n! \Gamma\left( m + n + l + 2 \right) \Gamma\left(l + 1  \right) } x^m y^n w^l.
\end{align*}
\noindent Hence
\begin{align*}
&S_3= z S - \frac{1}{w} \sum_{m,n=0}^{\infty}  \frac{  \Gamma\left( m + n + 3  \right) }{ m! n! } x^m y^n \\ &+  \sum_{m,n=0}^{\infty}  \frac{ \Gamma\left( m + n + 3  \right) ((m+n) (\sqrt{1-4w}-1) + (m+n+1) 4w) }{ m! n! (m+n+2)2 w^2 (1-4w)^{3/2} ( \sqrt{1-4w} +1)^{m+n}  } (2x)^m (2y)^n .
\end{align*}
\noindent After a little calculations, we have
\begin{align*}
S_3=& z S + \frac{2}{w (x+y-1)^3} \\ &+  \frac{16 \left(\left(\sqrt{1-4 w}+1\right) (1-x)-2 w^2\right)}{(1-4 w)^{3/2} w \left(\sqrt{1-4 w}-2 x+1\right) \left(\sqrt{1-4 w}-2 x-2 y+1\right)^3} \\ &- \frac{16 \left(\sqrt{1-4 w} y+2 \sqrt{1-4 w}-2 x (x+y+1)+y+4\right)}{(1-4 w)^{3/2} \left(\sqrt{1-4 w}-2 x+1\right) \left(\sqrt{1-4 w}-2 x-2 y+1\right)^3},
\end{align*} This completes the proof of Lemma \ref{lemd3}.
\end{proof}

\section{\bf Computation of the kernel}

For Reinhardt domains it is a standard method for computing the Bergman kernel  to use series representation, since we can choose $\phi_{\alpha} (z) = \frac{z^{\alpha}}{ \|z^{\alpha}\|}.$ Put $\Phi_{\alpha}(\zeta)= z_1^{\alpha_1} z_2^{\alpha_2} z_3^{\alpha_3} $. It is well known, that function $f$ holomorphic in a Reinhardt domain $D \subset \mathbb{C}^n$  has a “global” expansion into a Laurent series $f(z)=\sum_{\alpha \in \mathbb{Z}^n} a_{\alpha} z^{\alpha}$, $z \in D$ (see Proposition 1.7.15 (c) in \cite{JP}). Moreover if $D \cap (\mathbb{C}^{j-1} \times \{0\} \times \mathbb{C}^{n-j} ) \neq \emptyset $, $j=1,\ldots, n$ then $a_{\alpha}=0$ for $\alpha \in \mathbb{Z}^n \setminus \mathbb{Z}^n_{+}$ (see Proposition 1.6.5 (c) in \cite{JP}). Therefore  $\{\Phi_{\alpha} \}$ such that each $\alpha_i \geq 0$ is a complete orthogonal set for $L^2(D_1)$ and  $L^2(D_2)$. \newline If $D$ is a Reinhardt domain, $f \in L^2_a(D) := \mathcal{O}(D) \cap L^2(D)$,  $f(z)=\sum_{\alpha \in \mathbb{Z}^n} a_{\alpha} z^{\alpha}$, then $\{  z^{\alpha} \colon \alpha \in \sum(f)  \} \subset L^2_a(D), $ where $\sum(f):=\{ \alpha \in \mathbb{Z}^n \colon a_{\alpha} \neq 0 \}$ (for proof see \cite{JP} p. 67). Thus it is easy to check, that the set $\{z_1^{\alpha_1} z_2^{\alpha_2} z_3^{\alpha_3}\colon   \alpha_2 \geq 0,  \alpha_3 \geq 0, \alpha_1 \geq -1  - \alpha_3\}$ is a complete orthogonal set for $L^2_a(D_3)$.

\begin{pr}\label{pr1} Let $\alpha_i \in \mathbb{Z}_{+}$ for $i=1,2,3$. Then, we have
$${\left\| z_1^{\alpha_1} z_2^{\alpha_2} z_3^{\alpha_3}\right\|}^2_{L^2(D_1)}= \frac{ \pi^3 \Gamma(\alpha_1 + 1) \Gamma( \alpha_2 + \alpha_3 +3)  }{ (\alpha_2 + 1) (\alpha_3 +1 ) \Gamma( \alpha_1 + \alpha_2 + \alpha_3 + 4)} $$
\end{pr}

\begin{pr}\label{pr2} Let $\alpha_i \in \mathbb{Z}_{+}$ for $i=1,2,3$. Then, we have
$${\left\| z_1^{\alpha_1} z_2^{\alpha_2} z_3^{\alpha_3} \right\|}^2_{L^2(D_2)}= \frac{ 2 \pi^3 \Gamma(\frac{\alpha_1 + 1}{3}) \Gamma( \alpha_2 + \alpha_3 +3)  }{ 3 (\alpha_2 + 1) (\alpha_3 +1 ) \Gamma( \frac{\alpha_1 + 1}{3} + \alpha_2 + \alpha_3 + 3)} $$
\end{pr}

\begin{pr}\label{pr3} For $\alpha_2, \alpha_3 \in \mathbb{Z}_{+}$ and $\alpha_1 \geq -1 - \alpha_3$, we have
$${\left\| z_1^{\alpha_1} z_2^{\alpha_2} z_3^{\alpha_3} \right\|}^2_{L^2(D_3)}= \frac{  \pi^3 \Gamma(\alpha_1 + \alpha_3 + 2) \Gamma( \alpha_2 + \alpha_3 +3)  }{ (\alpha_2 + 1) (\alpha_3 +1 ) \Gamma( \alpha_1 + \alpha_2 + 2\alpha_3 + 5)} .$$
\end{pr}

\begin{pr}\label{pr4} Let $\alpha_i \in \mathbb{Z}_{+}$ for $i=1,2,3,4$. Then, we have
$${\left\| z^{\alpha}\right\|}^2_{L^2(D_4)}= \frac{  \pi^4 \Gamma(\alpha_1+1 ) \Gamma(\alpha_2 +1) \Gamma(\alpha_1 + \alpha_2 + \alpha_4 + 3) \Gamma( \alpha_3 + \alpha_4 +3)  }{ (\alpha_3 + 1) (\alpha_4 +1 )  \Gamma( \alpha_1 + \alpha_2 + 2) \Gamma( \alpha_1 + \alpha_2 + \alpha_3  +2\alpha_4 + 6)} .$$
\end{pr}

The proof of  all above propositions is similar, and so we only proof  \ref{pr2}.

\begin{proof}
$$\|  z_1^{\alpha_1} z_2^{\alpha_2} z_3^{\alpha_3}  \|^2_{L^2(D_2)} = \int\limits_{D_2} |z_1|^{2\alpha_1} |z_2|^{2\alpha_2} |z_3|^{2\alpha_3}dV(z)$$
we introduce polar coordinate in each variable by putting $z_1=r_1 e^{i\theta_1}$, $z_2=r_2 e^{i\theta_2}$, $z_3=r_3 e^{i\theta_3}$. After doing so, and integrating out the angular variables we have

$$(2 \pi)^3 \int_0^1 \int_0^{\sqrt{1-r_1^6}} \int_0^{\sqrt{1-r_1^6}} r_1^{2\alpha_1 + 1} r_2^{2\alpha_2 + 1} r_3^{2\alpha_3 + 1}\, dr_1 dr_2 dr_3$$
Integrating out of $r_2$ and $r_3$ variables, we obtain

$$ \frac{(2 \pi)^3}{(2\alpha_2 + 2 ) (2\alpha_3 + 2) } \int_0^1  r_1^{2\alpha_1 + 1} (1-r_1^6)^{\alpha_2 + \alpha_3 + 2} \, dr_1$$
After little calculation using well known fact $$\int_0^1 x^a(1-x^6)^b \,dx = \frac{\Gamma((a+1)/6) \Gamma(b+1)}{3  \Gamma((a+1)/6 + b + 1) } ,$$  we obtain desired result.

\end{proof}

Now by series representation of the Bergman kernel function, we have

$$K_{D_1} (z,w) = \frac{1}{\pi^3} \sum_{\alpha_1, \alpha_2 , \alpha_3 =0 }^{ \infty} \frac{(\alpha_2 + 1) (\alpha_3 +1 ) \Gamma( \alpha_1 + \alpha_2 + \alpha_3 + 4) }{\Gamma(\alpha_1 + 1) \Gamma( \alpha_2 + \alpha_3 +3)} \nu_1^{\alpha_1} \nu_2^{\alpha_2} \nu_3^{\alpha_3} ,$$
where $\nu_1= z_1 \overline{w}_1$, $\nu_2= z_2 \overline{w}_2$, $\nu_3= z_3 \overline{w}_3$. Sum out of $\nu_1$ variable, we have

$$\frac{1}{\pi^3 (1-\nu_1)^4} \sum_{\alpha_2, \alpha_3 =0 }^{ \infty} \frac{(\alpha_2 + 1) (\alpha_3 +1 ) \Gamma(  \alpha_2 + \alpha_3 + 4) }{ \Gamma( \alpha_2 + \alpha_3 +3)} \left( \frac{\nu_2}{1- \nu_1 } \right)^{\alpha_2} \left( \frac{\nu_3}{1- \nu_1 } \right)^{\alpha_3} $$

Since $\Gamma(  \alpha_2 + \alpha_3 + 4) = (\alpha_2 + \alpha_3 +3) \Gamma( \alpha_2 + \alpha_3 +3)$, that

$$\frac{1}{\pi^3 (1-\nu_1)^4} \sum_{\alpha_2, \alpha_3 =0 }^{ \infty} (\alpha_2 + 1) (\alpha_3 +1 ) (  \alpha_2 + \alpha_3 + 3)  \left( \frac{\nu_2}{1- \nu_1 } \right)^{\alpha_2} \left( \frac{\nu_3}{1- \nu_1 } \right)^{\alpha_3} $$

Finally using

$$ \sum_{n, k =0 }^{ \infty} (n + 1) (k +1 ) ( n+ k + 3)  x^{n} y^{k} = \frac{3 - x -y -xy}{(1-x)^3 (1-y)^3}$$

we obtain explicit formula for domain $D_1$.  Similarly  we can obtain  Bergman kernel for domains

$$D_{p,q}= \{ (z_1, z_2, z_3) \in \mathbb{C}^3 : |z_1|^2 + |z_2|^p <1 , \quad  |z_1|^2 + |z_3|^q < 1\}, $$
where $p,q$ are positive reals numbers. \newline Now we consider domain $D_2$. As before by \ref{pr2}, we have

$$K_{D_2} (z,w) =  \sum_{\alpha_1, \alpha_2 , \alpha_3 =0 }^{ \infty} \frac{3(\alpha_2 + 1) (\alpha_3 +1 ) \Gamma( \frac{\alpha_1 + 1}{3} + \alpha_2 + \alpha_3 + 3) }{ 2 \pi^3 \Gamma(\frac{\alpha_1 + 1}{3}) \Gamma( \alpha_2 + \alpha_3 +3) } \nu_1^{\alpha_1} \nu_2^{\alpha_2} \nu_3^{\alpha_3} ,$$
where $\nu_1= z_1 \overline{w}_1$, $\nu_2= z_2 \overline{w}_2$, $\nu_3= z_3 \overline{w}_3$.

Using partial derivative notation we can write

$$ \frac{3}{2 \pi^3} \frac{\partial^2}{ \partial \nu_2 \partial \nu_3} \sum_{\alpha_1, \alpha_2 , \alpha_3 =0 }^{ \infty} \frac{ \Gamma( \frac{\alpha_1 + 1}{3} + \alpha_2 + \alpha_3 + 3) }{ \Gamma(\frac{\alpha_1 + 1}{3}) \Gamma( \alpha_2 + \alpha_3 +3) } \nu_1^{\alpha_1} \nu_2^{\alpha_2 + 1} \nu_3^{\alpha_3 + 1} ,$$

Now we can express above sum in $F_8$ hypergeometric function terms separating $\alpha_1$ modulo 3.

\begin{align*}
\frac{3}{2 \pi^3 } \frac{\partial^2}{\partial \nu_2 \partial \nu_3} \left\{\nu_2 \nu_3 \sum_{j=1}^3 \nu_1^{j-1} C_j F_8\left(3+\frac{j}{3},1,1,1;\frac{j}{3},3; \nu_1^3,\nu_2,\nu_3\right) \right\} ,
\end{align*}

where  $C_i=\frac{\Gamma\left(3+\frac{i}{3}\right)}{2 \Gamma\left(\frac{i}{3}\right) }$ for $i=1,2,3$. After some calculations using explicit formulas from lemma \ref{lem1} we obtain desired result. \newline It is also possible in analogous way, compute Bergman kernel function for $$ \{(z_1, z_2, z_3) \in \mathbb{C}^3 : |z_1|^r + |z_2|^2 <1 , \quad  |z_1|^2 + |z_3|^2 < 1\} ,$$ for any rational number $r$ and  for $$ \{(z_1, z_2, z_3) \in \mathbb{C}^3 : |z_1|^4 + |z_2|^4 <1 , \quad  |z_1|^4 + |z_3|^4 < 1\} .$$

Now we consider domain $D_3$. From Proposition \ref{pr3}, we have

\begin{align*}
K_{D_3} (z,w) =\sum_{\alpha_2, \alpha_3 , \alpha_1 =0 }^{ \infty} \frac{(\alpha_3+1) (\alpha_2+1)  \Gamma (\alpha_1+\alpha_2+\alpha_3+4)}{ \pi^3 \nu_1 \Gamma (\alpha_1+1) \Gamma (\alpha_2+\alpha_3+3)}  \nu_1^{\alpha_1} \nu_2^{\alpha_2} \left(\frac{\nu_3}{\nu_1}\right)^{\alpha_3}
\end{align*}
Sum out of $\nu_1$ variable, we have

\begin{align*}
 \sum_{\alpha_2 , \alpha_3 =0 }^{ \infty} \frac{(\alpha_3+1) (\alpha_2+1)  \Gamma (\alpha_2+\alpha_3+4)}{(1-\nu_1)^4 \pi^3 \nu_1  \Gamma (\alpha_2+\alpha_3+3)}  \left( \frac{\nu_2}{1-\nu_1} \right)^{\alpha_2} \left(\frac{\nu_3}{(1-\nu_1)\nu_1}\right)^{\alpha_3}
\end{align*}

Sum out of $\nu_2$ variable, we get

\begin{align*}
 \sum_{ \alpha_3 =0 }^{ \infty} \frac{(\alpha_3+1) (\alpha_3 \nu_1 + \alpha_3 \nu_2 - \alpha_3 + 3 \nu_1+\nu_2-3)  }{\pi^3 \nu_1 (1-\nu_1)^2 (\nu_1+\nu_2-1)^3}   \left(\frac{\nu_3}{(1-\nu_1)\nu_1}\right)^{\alpha_3}
\end{align*}

Finally sum out of $\alpha_3$ variable, we obtain the desired result.

In order to prove formula for domain $D_4$ is sufficient to use lemma \ref{lemd3}.

\section{\bf Lu Qi-Keng’s problem}

A domain $\Omega \subset \mathbb{C}^n$ is called a Lu Qi-Keng domain if $K_{\Omega}(z, \overline{w}) \neq 0$ for all $z, w \in \Omega$. Obviously, a biholomorphic image of a Lu Qi-Keng domain is a Lu Qi-Keng domain due to the rule of the Bergman kernel transformation between two biholomorphic equivalent domains. A Cartesian product of two Lu Qi-Keng domains is a Lu Qi-Keng domain. If $K_{\Omega} \neq const$ and $\Omega$ is the sum of an increasing sequence of Lu Qi-Keng domains $\Omega_m$, then $\Omega$ is a Lu Qi-Keng domain due to the Ramadanov theorem and Hurwitz theorem. However, it is not always easy to determine whether or not a given domain is Lu Qi-Keng domain. In 1969, M. Skwarczynski \cite{SKWA} gave the first example that the Bergman kernel on an annulus in the complex plane $\Omega = \{r < |z| < 1\}$ has zeros if $0 < r < e^{-2}$. Since the Bergman kernel for a bounded symmetric domain is a negative power of a certain polynomial, it has no zeros anywhere. In 1996, Boas \cite{Boas1} proved that bounded Lu Qi-Keng domains of holomorphy in $\mathbb{C}^n$ form a nowhere dense subset of all bounded domains of holomorphy. Since then the concrete forms of non-Lu Qi-Keng domains have been found in the various classes of domains in $\mathbb{C}^n$. The minimal ball \cite{OPY} with $n \geq 4$ and the symmetrized polydisks \cite{NZ} with $n \geq 3$ are not Lu Qi-Keng domains. For the complex ellipsoids, see \cite{BS, ZY}. \newline The explicit formula of the Bergman kernel function for the domain $D$ enables us to investigate
whether the Bergman kernel has zeros in $D \times D$ or not. We will call this kind of problem a Lu Qi-Keng problem. The motivation of this problem comes from the Riemann mapping theorem. If $n \geq  2$, then there is no analogue of Riemann mapping theorem in $\mathbb{C}^n$. Thus the following natural question arises: Are there canonical representatives of biholomorphic equivalence classes of domains? In higher dimensions, Bergman himself \cite{BE} introduced a representative
domain to which a given domain may be mapped by representative coordinates. Let $K(z,w)$ be the Bergman kernel for a bounded domain $D \in \mathbb{C}^n$, and define $$T_{i \overline{j}}(z)=(g_{ij}) := \frac{\partial^2  }{ \partial \zeta_i \partial \overline{\zeta}_j} \log K(\zeta,\zeta)_{|\zeta=z}.$$ Then its converse is $T^{  \overline{j} i} (z)=(g_{ji}^{-1})$. Hence the local representative coordinate $f (z) =(f_1, \ldots , f_n)$ based at the point $z_0$ is given by $$f_i(z)=\sum_{i=1}^n T^{  \overline{j} i} (z_0) \frac{\partial}{ \partial \overline{\zeta}_j} \log \frac{K(z,\zeta)}{K(\zeta,\zeta) }{\Big|_{\zeta=z_0}}$$ for $i = 1, \ldots , n.$
In 1966, Lu Qi-Keng \cite{Qi} observed the following phenomenon: It is necessary that the Bergman kernel $K(z,w)$ has no zeros in order to define the Bergman representative coordinates.

\begin{lemma}
The domains  $D_1$ and $D_3$ are Lu Qi-Kengs domains.
\end{lemma}

\begin{proof}
 Suppose that the Bergman  function for $D_1$ is zero at $(z, w)$.  It means that there exists $z=(z_1,z_2,z_3) \in D_1$ and $w=(w_1,w_2,w_3) \in D_1$, such that $K_{D_1}(z,w)=0$. Then by Theorem \ref{th1}, we have
$$2\nu_1^4 = (\nu_1^2 \nu_3 + \nu_1^3)(\nu_1^2 + \nu_2),$$ where $\nu_i=z_i \overline{w}_i$. Since
\begin{eqnarray*}
|z_3|^2 < |z_1|^4 + |z_2|^2, \quad |z_1|^4 + |z_2|^2 < |z_1|^2 \\ |w_3|^2 < |w_1|^4 + |w_2|^2, \quad |w_1|^4 + |w_2|^2 < |w_1|^2,
\end{eqnarray*}
then $|\nu_3|< |\nu_1|$. Hence
\begin{equation}\label{lu1}
2|\nu_1|^4 = |(\nu_1^2 \nu_3 + \nu_1^3)(\nu_1^2 + \nu_2)| < 2|\nu_1|^3(|\nu_1|^2 + |\nu_2|).
\end{equation}
from the Cauchy-Schwarz inequality, easily follows
$$(|\nu_1|^2 + |\nu_2|)^2 \leq (|z_1|^4 + |z_2|^2)(|w_1|^4 + |w_2|^2 )$$ Hence $|\nu_1|^2 + |\nu_2| < |\nu_1|$,  which is a contradiction with (\ref{lu1}). \newline
\indent  In order to prove lemma for domain $D_3$ is sufficient to use observation, that  function defined by $$F \colon D_1 \setminus \{0\} \mapsto D_3, \quad F(z_1,z_2,z_3)=(z_1,z_2,z_1 z_3)$$ is biholomorphic between domains $D_1 \setminus \{0\}$ and $D_3$.
\end{proof}
 It is well known fact, that Bergman kernel function for unit ball in $\mathbb{C}^2$ is zero free. In view of the above lemma, we can ask the following question: Is there a relationship between the existence of zeros of  the Bergman kernel function for  domains $$\{z \in \mathbb{C}^2 \colon |z_1|^p + |z_2|^q < 1\}\quad  \text{and} \quad \{z \in \mathbb{C}^3 \colon |z_1|^p + |z_2|^q < 1, \, |z_1|^p + |z_3|^q < 1 \}\,?$$


\begin{thebibliography}{99}


\bibitem{AK} P. Appell, J. Kamp\'{e} de F\'{e}riet, \textit{Fonctions hyperg\'{e}om\'{e}trigues et hypersph\'{e}riques}, Gauthier-Villars, Paris, 1926.

\bibitem{Bell} S. Bell, \textit{The Bergman kernel function and proper holomorphic mappings}, Trans. Amer. Math. Soc., 270(2) (1982) 685-691.

\bibitem{BE} S. Bergman, \textit{ Zur Theorie von pseudokonformen Abbildungen}, Mat. Sb. (N.S.) 1(43) (1) (1936), 79-96.

\bibitem{Bers} L. Bers, \textit{Mathematical Aspects of Subsonic and Transonic Gas Dynamics}, Wiley, New York, 1958.

\bibitem{B}  Z. B{\l}ocki,  \textit{Some estimates for the Bergman kernel and metric in terms of logarithmic capacity}, Nagoya Math. J., 185 (2007), 143-150.

\bibitem{Boas} H.P. Boas, \textit{Counterexample to the Lu Qi-keng conjecture}, Proc. Amer. Math. Soc., 97 (1986),
374-375.

\bibitem{BL} H.P. Boas, \textit{Lu Qi-Keng's problem}, J. Korean Math. Soc., 37(2) (2000), 253-267.

\bibitem{Boas1} H.P. Boas, \textit{The Lu Qi-Keng conjecture fails generically},  Proc. Amer. Math. Soc. 124(7) (1996), 2021-2027.

\bibitem{BS} H.P. Boas, S. Fu, E.J. Straube, \textit{The Bergman kernel function: explicit formulas and zeroes}, Proc. Amer. Math. Soc., 127(3) (1999), 805-811.

\bibitem{decH3} J. Choi, A. Hasanov i M. Turaev,  \textit{Decomposition formulas and integral representations for some Exton hypergeometric functions}, J. Chungcheong Math. Soc., 24(4) (2011), 745-758.

\bibitem{C} D. Cvijovi\'{c}, \textit{Polypseudologarithms revisited}, Physicia A,  389 (2010),1594-1600.

\bibitem{DA1} J.P. D'Angelo, \textit{A note on the Bergman kernel}, Duke Math. J., 45 (1978) 259-265.
\bibitem{DA2} J.P. D'Angelo, \textit{An explicit computation of the Bergman kernel function}, J. Geom. Anal., 4 (1994) 23-34.

\bibitem{EZ} A. Edigarian, W. Zwonek, \textit{Geometry of the symmetrized polydisc}, ARCH MATH,  84 (2005), 364-374.

\bibitem{Edr} A. Erd\'{e}lyi, W. Magnus, F. Oberhettinger i FG. Tricomi, \textit{Higher Transcendental Functions}, vol. I. McGraw-Hill, New York, 1953.

\bibitem{H1} A. Hasanov, \textit{ Fundamental solutions of generalized bi-axially symmetric Helmholtz equation.} Complex Variables and Elliptic Equations, 52(8) (2007), 673-683.

\bibitem{H2} A. Hasanov, \textit{The solution of the Cauchy problem for generalized Euler-Poisson-Darboux equation.}
International Journal of Applied Mathematics and Statistics, 8(7) (2007), 30-44.

\bibitem{JP} M. Jarnicki, P. Pflug, \textit{First steps in several complex variables: Reinhardt domains}, European Mathematical Society, 2008

\bibitem{FH2} G. Francsics, N. Hanges, \textit{Asymptotic behavior of the Bergman kernel and hypergeometric functions}, in: Multidimensional Complex Analysis and Partial Differential Equations (So Carlos, 1995), 7992, in: Contemp. Math., vol. 205, Amer. Math. Soc, Providence, RI, 1997.

\bibitem{FH}  G. Francsics, N. Hanges,  \textit{The Bergman kernel of complex ovals and multivariable hypergeometric functions}, J. Funct. Anal., 142 (1996), 494-510.


\bibitem{Fran} F.I. Frankl, \textit{Selected Works in Gas Dynamics}, Nauka, Moscow 1973, (in Russian).


\bibitem{Hua} L.K. Hua,
\textit{Harmonic Analysis of Functions of Several Complex Variables in Classical Domain}, in: Translation of Math. Monographs, vol. 6, Amer. Math.
Soc., Providence, Rhode Island, Moskow, 1979 (in Russian).
\bibitem{L} E. Ligocka, \textit{On the Forelli-Rudin construction and weighted Bergman projection}, Studia Math., 94 (1989), 257-272.
\bibitem{LO} G. Lohofer, \textit{Theory of an electromagnetically deviated metal sphere.} 1: Absorbed power. SIAM J.
Appl. Math., 49 (1989), 567-581.

\bibitem{Qi} Qi Keng Lu. \textit{On K\"{a}hler manifolds with constant curvature}, Chinese Math.-Acta, 8 (1966), 283-298.

\bibitem{MA} A.M. Mathai, R. K. Saxena, \textit{Generalized Hypergeometric Functions with Applications in Statistics
and Physical Sciences.} Springer-Verlag, Berlin, Heidelberg and New York, 1973.

\bibitem{NI} A.W. Niukkanen, \textit{Generalised hypergeometric series $^NF (x_1, ..., x_N)$ arising in physical and quantum chemical applications}, J. Phys. A: Math. Gen., 16 (1983), 1813-1825.

\bibitem{NZ} N. Nikolov, W. Zwonek, \textit{The Bergman kernel of the symmetrized polydisc in higher dimensions has zeros},
Arch. Math. (Basel) 87(5) (2006), 412-416.

\bibitem{OPY} K. Oeljeklaus, P. Pflug i E.H. Youssfi, \textit{The Bergman kernel of the minimal ball and applications}, Ann. Inst. Fourier (Grenoble), 47(3) (1997), 915-928.

\bibitem{OSS} S.B. Opps, N. Saad i H.M. Srivastava, \textit{Some reduction and transformation formulas for the Appell hypergeometric function $F_2$}, J. Math. Anal. Appl., 302 (2005), 180-195.

\bibitem{Park1} J.-D. Park, \textit{New formulas of the Bergman kernels for complex ellipsoids in $C^2$}, Proc. Amer. Math. Soc., 136(12) (2008), 4211-4221.
\bibitem{Park2} J.-D. Park, \textit{Explicit formulas of the Bergman kernel for 3-dimensional complex ellipsoids},  J. Math. Anal. Appl., 400(2) (2013), 664-674.


\bibitem{SKWA} M. Skwarczynski, \textit{The distance in theory of pseu-conformal transformations and the Lu Qi-Keng conjecture}. Proc. Amer. Math. Soc., 22 (1969), 305-310.

\bibitem{Multiple} H.M. Srivastava, P.W. Karlsson, \textit{Multiple Gaussian Hypergeometric Series}, Halsted Press (Ellis Horwood Limited, Chichester), Wiley, New York, Chichester, Brisbane and Toronto, 1985.

\bibitem{SE} I.N. Sneddon, \textit{Special Functions of Mathematical Physics and Chemistry.} Third ed., Longman, London, New York, 1980.

\bibitem{ZY} L. Zhang, W. Yin, \textit{Lu Qi-Keng's problem on some complex ellipsoids}, J. Math. Anal. Appl., 357(2) (2009), 364-370.

\end{thebibliography}
\end{document}